\documentclass[12pt]{article}

\usepackage{amsmath,amsthm}
\usepackage{amssymb}
\usepackage{amsfonts}
\usepackage{amscd}
\usepackage{booktabs}
\usepackage{slashbox}
\usepackage{dsfont} 

\usepackage{mathtools} 
\input{xy}
\xyoption{all}

\newtheorem{thm}{Theorem}[section]
\newtheorem{theorem}[thm]{Theorem}

\newtheorem{corollary}[thm]{Corollary}

\newtheorem{proposition}[thm]{Proposition}
\theoremstyle{definition}
\newtheorem{definition}[thm]{Definition}
\newtheorem{example}[thm]{Example}

\newtheorem{remark}[thm]{Remark}

\newtheorem{observation}[thm]{Observation}
\newtheorem{definition-proposition}[thm]{Definition-Proposition}

\newcommand{\R} {\mathbf{R}}
\newcommand{\Z} {\mathbf{Z}}

\newcommand{\C} {\mathbf{C}}

\newcommand{\bA}{\mathbf{A}}

\newcommand{\bm} {\mathbf {m}}
\newcommand{\bp} {\mathbf {p}}

\newcommand{\cJ}{{\mathcal J}}

\newcommand{\cO}{\mathcal{O}}

\newcommand{\cU}{\mathcal{U}}
\newcommand{{\cHol}}{{\cal{{H}}}}

\newcommand{\al}{\alpha}
\newcommand{\be}{\beta}
\newcommand{\ga}{\gamma}

\newcommand{\fg}{\mathfrak g}

\newcommand{\fh}{\mathfrak h}
\newcommand{\fk}{\frak k}

\newcommand{\bt}{\overline t}
\newcommand{\bu}{\overline u}
\newcommand{\bv}{\overline v}

\newcommand{\salg} {\mathrm{ (salg) }}
\newcommand{\sets} {\mathrm{ (sets) }}

\newcommand\sschemes{ \mathrm{(sschemes)} }

\newcommand{\grps}{ \mathrm{(groups) }}

\newcommand{\rGL}{\mathrm{GL}}

\newcommand{\rSL}{\mathrm{SL}}

\newcommand{\rosp}{{\mathrm{osp}}}

\newcommand{\Hom}{\mathrm{Hom}}
\newcommand{\rEnd}{{\hbox{End}}}

\newcommand{\rsl}{{\mathrm{sl}}}

\newcommand{\Ad}{\mathrm{Ad}}

\newcommand{\Lie}{\mathrm{Lie}}

\newcommand{\uspec}{\mathrm{\underline{Spec}}}

\newcommand{\btheta}{\bar{\theta}}
\newcommand{\etab}{\bar{\eta}}

\usepackage{color}

\let\mto\mapsto \renewcommand{\mapsto}{\mathchoice{\longmapsto}{\mto}{\mto}{\mto}} 

\newcommand{\defi} {\mathrm{def}}
\newcommand{\lra} {\longrightarrow}

\newcommand{\spec}{{\mathrm{Spec}}}

\newcommand{\rd}{{\mathrm{red}}}
\newcommand{\rsu}{{\mathrm{su}}}
\newcommand{\bs}{{\overline{s}}}

\begin{document}

\Large
\centerline{\bf
{Compact forms of Complex Lie Supergroups}}

\normalsize
\bigskip

\centerline{ R. Fioresi
}

\smallskip

\centerline{\it Dipartimento di Matematica, Universit\`{a} di
Bologna }
 \centerline{\it Piazza di Porta S. Donato, 5. 40126 Bologna. Italy.}
\centerline{{\footnotesize e-mail: rita.fioresi@UniBo.it}}

\bigskip

\begin{abstract}
In this paper we construct compact forms associated with a 
complex Lie supergroup with Lie superalgebra of classical type.
\end{abstract}

\medskip

\section{Introduction}

In the ordinary theory of semisimple Lie groups, we can associate
to any semisimple complex Lie group a corresponding real compact form.
It is then very natural to ask, whether the same statement is true
in the context of Lie supergroups. Given the rigidity of supergeometry
requirements, it turns out that not all of the Lie supergroups with
Lie superalgebras of classical type admit such a compact form, however
under some hypothesis it is possible to grant such existence.

\medskip

Our starting point is a complex analytic supergroup $G$ with $\fg=\Lie(G)$ of
classical type. Such $G$ corresponds to a unique
affine algebraic supergroup constructed via the Chevalley recipe
as detailed in \cite{fg}, \cite{fg2}.
Once a suitable real form $\fk$ of $\fg$ 
is obtained through an involution, we proceed and build the corresponding
compact supergroup $K$, with $\fk$ as its Lie superalgebra.
The superalgebra of global sections of $K$ is precisely given
by the global sections of $G$ invariant under the involution
defined through the involution of $\fg$ defining $\fk$.
We repeat the same construction using the Super Harish-Chandra Pairs
(SHCP) terminology establishing a complete equivalence between the
two different approaches to this problem.

\bigskip

{\bf Acknowledgements}. We wish to thank prof. V. S. Varadarajan,
prof. C. Carmeli and prof. F. Gavarini
for helpful comments. We are also indebted to our referee for
very valuable suggestions.

\section{Preliminaries}
\label{preliminaries}

We review few facts about supergeometry, sending the reader to
\cite{ccf, vsv} for more details.

\medskip

Let $k=\R$ or $\C$ be our ground field.

\medskip

A {\it superalgebra} $A$ is a $\Z_2$-graded algebra, $A=A_0 \oplus A_1$,
where  $p(x)$ denotes the parity of an homogeneous element $x$.
The superalgebra $A$ is said to be {\it commutative} if for any two
homogeneous elements $x, y$, $xy=(-1)^{p(x)p(y)}yx$.
All of our superalgebras are assumed to be commutative
unless otherwise specified and their
category will be denoted by $\salg_k$.

\medskip

\begin{definition}
A {\it superspace} $S=(|S|, \cO_S)$ is a topological space $|S|$
endowed with a sheaf of superalgebras $\cO_S$ such that the stalk at
a point  $x\in |S|$ denoted by $\cO_{S,x}$ is a local superalgebra
for all $x \in |S|$, i. e. it has a unique (two-sided) ideal.
A {\it morphism} $\phi:S \lra T$ of superspaces is given by
$\phi=(|\phi|, \phi^*)$, where $\phi: |S| \lra |T|$ is a map of
topological spaces and $\phi^*:\cO_T \lra \phi_*\cO_S$ is
a local sheaf morphism, that is,
$\phi_x^*(\bm_{|\phi|(x)})=\bm_x$, where $\bm_{|\phi|(x)}$ and
$\bm_{x}$ are the maximal ideals in the stalks $\cO_{T,|\phi|(x)}$
and $\cO_{S,x}$ respectively.

\medskip

If $\cJ_S$ is the ideal sheaf generated by the odd sections
in $\cO_S$, we define the ordinary ringed space
$S_0=(|S|, \cO_S/\cJ_S)$ as the \textit{ringed space underlying} 
the superspace $S$.
\end{definition}

The most important examples of superspaces are given by
{\sl supermanifolds} and {\sl superschemes}. 

\begin{example} \label{superspace}
1. Let $k=\R$ or $\C$.
The superspace $k^{p|q}$ is the topological space $k^p$ endowed with
the following sheaf of superalgebras. For any $U
\subset_\mathrm{open} k^p$
$$
\cO_{k^{p|q}}(U)=\cO_{k^p}(U)\otimes k[\xi_1, \dots, \xi_q],
$$
where $k[\xi_1, \dots ,\xi_q]$ is the exterior algebra 
generated by the $q$ variables $\xi_1, \dots,
\xi_q$ and $\cO_{k^p}$ denotes the ordinary $C^\infty$ or analytic
sheaf over $k^p$.

\medskip

2. $\uspec A$. 
Let $A$ be an object of $\salg_k$. Since $A_0$ is an algebra, we can
consider the topological space $$\spec(A_0)=\{\hbox{prime ideals }
\bp\subset A_0\}$$  with its structural sheaf $\cO_{A_0}$. The stalk $A_{\bp}$
of the structural sheaf at the prime $\bp\in \spec(A_0)$ is the
localization of $A_0$ at $\bp$.
As for any superalgebra, $A$ is a module over $A_0$, and we have
a sheaf $\widetilde A$ of $\cO_{A_0}$-modules over $\spec A_0$
with stalk $A_{\bp}$, the localization of the $A_0$-module $A$ over
each prime $\bp \in \spec(A_0)$.
$\uspec A=_{\defi}(\spec
A_0,\widetilde A)$ is a superspace.
\end{example}

\begin{definition}\label{supermanifold}
A  \textit{differentiable (resp. analytic) supermanifold} 
of dimension $p|q$ is a superspace $M=(|M|, \cO_M)$
which is locally isomorphic to the superspace
$k^{p|q}$, where the ordinary sheaf over $k^p$ is
taken to the be the $C^\infty$ or the analytic sheaf
respectively in case of $k=\R$ or $k=\C$.

\medskip

An \textit{affine superscheme} is a superspace isomorphic to $\uspec A$
for a superalgebra $A$.
A  \textit{superscheme} is a superspace 
which is locally isomorphic to affine superschemes.

\end{definition}

Morphisms  of supermanifolds  or of superschemes are just the
morphism of the corresponding superspaces.

\begin{example} \label{affinesuperspace}
1. {\sl Affine superspace}. Define 
$$
\bA^{m|n}:=
\uspec k[x_1, \dots, x_m,\xi_1, \dots, \xi_m]
$$ 
(the odd variables $\xi_1 \dots \xi_n$ anticommute). This superscheme
is called the \textit{affine superspace} of dimension $m|n$.
Its underlying scheme is the affine space $\bA^m$ of dimension $m$.

\medskip

2. {\sl General linear supergroup}. Let $V=V_0 \oplus V_1$ be
a finite dimensional vector superspace.
Define 
$$
\rGL(V)=(|\rGL(V_0)| \times  |\rGL(V_1)| , 
\cO_{\rEnd(V)}|_{|\rGL(V_0)| \times  |\rGL(V_1)|})
$$ 
where $\rEnd(V)$
is the super vector space of endomorphism of $V$ and $\cO_{\rEnd(V)}$
denotes the supersheaf of differential (resp. analytic or algebraic)
sections on $\rEnd(V)$.
Since $|\rGL(V_0)| \times  |\rGL(V_1)|$ is open in $|\rEnd(V)|$
we have a superspace, which is a supermanifold or an affine superscheme
resp., depending on whether we take $\rEnd(V)$ as in Ex. \ref{superspace} (1)
or as in Ex. \ref{affinesuperspace}, (1).
\end{example}

Next we want to introduce the concept of {\it functor
of points} of a superscheme (resp. supermanifold).

\begin{definition}
The {\it functor of points} of a
superscheme $X$ is a representable functor from the
category of superschemes to the category of sets:
$$
\begin{CD}
h_X:\sschemes^o @>>>\sets \\
T @>>>h_X(T)=\Hom(T,X)
\end{CD}
$$
and $h_X(\phi)f=f \circ \phi$ for any morphism $\phi:T \lra S$.
The elements in  $h_X(T)$ are called the \textit{$T$-points} of $X$.
(The label `$\,{}^o\,$' means that we are taking the opposite category.)
\end{definition}

In the very same way it is possible to define the functor of points
of a supermanifold. 

\medskip

It is customary to use the same letter to denote a supermanifold (or
a superscheme) and its functor of points. We shall use two different
symbols $X$ and $h_X$ only when strictly necessary.

\begin{observation} \label{facts}
The functor of points of a superscheme is determined by
its restriction to the category of affine superschemes, which is equivalent
to the category of affine superalgebras $:\salg_k$ (see \cite{ccf} ch. 10).
Hence 
given a superscheme $X$, we can
equivalently define its functor of points as:
$$
\begin{CD}
h_X:\salg_k @>>>\sets \\
A @>>>h_X(A)=\Hom(\uspec A, X)
\end{CD}
$$
and $h_X(\phi)f=f \circ \uspec \phi$ for any morphism $\phi:A \lra B$.
The elements in  $h_X(A)$ are called the \textit{$A$-points} of $X$.
When $X$ is itself an affine superscheme, its functor of points is
{\sl representable}:
$$
\begin{CD}
h_X:\salg_k @>>>\sets \\
A @>>>h_X(A)=\Hom(\cO(X), A)
\end{CD}
$$
and $h_X(\phi)f=\phi \circ f$ for any morphism $\phi:A \lra B$.
$\cO(X)$ denotes the superalgebra of global section of $\cO_X$.
The functor of points of a superscheme $F:\salg_k
\lra \sets$ is a local functor i. e. it has the sheaf property.
In other words let $A \in \salg_k$ and
$(f_i,\; i \in I)=(1)=A$. Let
$\phi_{k}:A \lra A_{f_k}$ be the natural map of the algebra $A$ to its
localization and also $\phi_{kl}:A_{f_k} \lra A_{f_kf_l}$.
Then, for a family
$\al_i\in F(A_{f_i})$, such that $F(\phi_{ij})(\al_i)=F(\phi_{ji})(\al_j)$
there exists $\al\in F(A)$ such that
$F(\phi_i)(\al)=\al_i$.
A generic functor $F: \salg_k \lra \sets$ is the functor of points
of a superscheme if and only if it is local and it admits a cover
by open affine subfunctors (i.e. the functor of points of an affine
superscheme) (ref. \cite{ccf} ch. 10, Representability Criterion).

\medskip

Finally we recall that by Yoneda's
lemma, given two superschemes (resp. supermanifolds) $S$ and $T$, the natural
transformations $h_S \lra h_T$ are in one-to-one correspondence with
the superscheme (resp. supermanifold) 
morphisms $S \lra T$. Consequently two superschemes (resp. supermanifolds) are
isomorphic if and only if their functor of points are isomorphic.

\end{observation}

\section{Analytic and Algebraic Supergroups} 
\label{analytic-algebraic}

A supermanifold (resp. affine superscheme) $G$ is a \textit{supergroup}
when its functor of points $h_G$ is valued in the category of groups; 
this is equivalent to say that the
superalgebra of global sections $\cO(G)$ has a Hopf superalgebra 
structure (see \cite{ccf,vsv} for more details on Lie, analytic and
algebraic supergroups). We say that the supergroup scheme $G$ is an 
\textit{algebraic supergroup} when
$G_0$ is algebraic and $\cO(G)$ is finitely generated as $\cO(G)_0$-module.

\medskip

A supergroup $G$ is called a \textit{matrix supergroup}
if $G$ admits an embedding into $\rGL(V)$, where we fix a homogeneous
basis for the super vector space $V$.

\medskip

From now on all of our supergroups are assumed to be matrix supergroups.

\medskip

We now want to establish a correspondence between complex analytic
Lie supergroups and algebraic supergroup schemes.
The next
observation is crucial in the developing of our work.

\begin{observation} \label{anal-alg}
If $G$ is an affine algebraic complex supergroup, $G$ has a unique analytic
supergroup structure. This happens in the same way as in the
ordinary setting and the details of the construction appear
in \cite{fi1}. Vice-versa if $G=(|G|,\cO_G)$ 
is a complex analytic matrix supergroup
(i.e. $G$ admits an embedding into some $\rGL(V)$),
there exists a unique algebraic complex supergroup $G_a$ associated with
it.  In fact the 
sheaf of superalgebras $\cO_G$ splits,
that is $\cO_G \cong \cO_{G_{\rd}} \otimes \wedge(\xi_1 \dots \xi_q)$,
where $G_{\rd}$ denotes the ordinary analytic group underlying $G$.
Since $G_{\rd}$ has a unique algebraic group $G_{a,\rd}$ associated with it,
we can define the sheaf of $G_a$ as 
$\cO_{G_a}=\cO_{G_{a,{\rd}}} \otimes \wedge(\xi_1 \dots \xi_q)$
and $G_a=\uspec \cO({G_a})$.

\medskip

Hence we can view the same supergroup $G$ as an analytic or an 
algebraic supergroup. We shall use the same letter to denote both,
when the context makes clear whether we are considering the
analytic of the algebraic structure on $|G|$.
\end{observation} 

Assume now that $G$ is a complex Lie supergroup, 
with Lie superalgebra $\fg$ of classical
type, that is $\fg$ is non abelian simple (no nontrivial ideals) and
$\fg_1$ is completely reducible as $\fg_0$-module. 
All of what we say holds also for $G$ with Lie superalgebra which is a
direct sum of superalgebras of classical type. We shall neverthless,
for clarity of exposition, examine only the case of $\fg$ of
classical type.
According to
\cite{kac} we have a complete list of such Lie superalgebras:  

\begin{theorem}
\label{classificationtheorem}
The complex Lie superalgebras of classical type  are either isomorphic 
to a simple Lie algebra or to one of the following 
Lie superalgebras:
$$  
\begin{array}{c}
A(m,n) \, ,  \;\; m \! \geq \! n \! \geq \! 0 \, , \, m+n > 0 \, ;  \quad  
B(m,n) \, ,  \;\; m \geq 0, n \geq 1 \, ;   \quad  C(n) \, ,  \;\; n \geq 3 
\\ \\
D(m,n) \, ,  \;\; m \geq 2 , \, n \geq 1 \, ;  \qquad  
P(n) \, ,  \;\; n \geq 2 \, ;   \quad  Q(n) \, ,  \;\; n \geq 2  
\\ \\
F(4) \; ;  \qquad  G(3) \; ;  \qquad  
D(2,1;a) \, ,  \;\; a \in k \setminus \{0, -1\}. 
\end{array} 
$$
\end{theorem}

\medskip

Any complex affine algebraic supergroup with Lie superalgebra of 
classical type can be built via the Chevalley's supergroups 
recipe (ref. \cite{fg2}).
Briefly, this amounts to choose a suitable complex representation $V$
of the complex Lie superalgebra $\fg$ of classical type with root system
$\Delta$. If $\salg_k$ denotes the category of commutative superalgebras,
the functor of points 
of the Chevalley supergroup $G:\salg_k \lra \sets$ is defined as:
$$
G(A)=\langle G_0(A), \, 1+\theta_\al X_\al, \, \al \in \Delta_1, \,
\theta_\al \in A_1 
\rangle \,
\subset \, \rGL(V)(A), \qquad A \in \salg_k
$$
where as usual we omit the representation and write just $X_\al$ for
the action of the root vector $X_\al$ on $V$ in a fixed Chevalley basis
(see \cite{fg} for the definition of Chevalley basis in the super
context). Since $G(A) \subset \rGL(V)(A)$, the definition on the
morphism is clear.

\medskip

In \cite{fg} we prove that such a functor is representable, so that
it is indeed an algebraic supergroup.

\section{Compact form of a Lie superalgebra} \label{compactform}

Assume $G$ is a complex analytic Lie supergroup, $\fg$ its Lie superalgebra,
$\fg$ of classical type.

\begin{definition} \label{defcompactliealg}
We say that the real Lie superalgebra $\fk \subset \fg$ is a \textit{compact
real form} of $\fg$ if
\begin{itemize}

\item $\fk_0$ is of compact type (i. e. $\Ad(\fk_0)$ is compact);  

\item $\C \otimes_\R \fk=\fg$.
\end{itemize}
We say that a real Lie supergroup $K$ is \textit{maximal compact} in $G$
if $\Lie(K)$ is a compact real form of $\fg$.
\end{definition}

As we will see in Obs. \ref{osp-counterex}, contrary
to what happens in the ordinary setting, the existence of compact real forms
is not granted for $\fg$ of every classical type.

\medskip

We can determine a compact real
form of a Lie superalgebra of classical type via an involution.
Let $\fh$ be a CSA of $\fg$, $\Delta=\Delta^+ \cup \Delta^-$ 
the root system. Let $\{H_j, X_\al\}$ be
a Chevalley basis for $\fg$; the existence of such basis is
granted for $\fg$, see \cite{fg}. Let $H_\al=[X_\al,X_{-\al}]$,
this is an integral combination of the $H_j$'s (see \cite{fg}).
\medskip

Assume
\begin{equation}
[X_\al, X_\al ]=0, \qquad \hbox{for all } \al \in \Delta  \label{assumption}
\end{equation}

\begin{theorem} \label{compactlie-thm}
Let $\fg$ be as above and 
let:
$$
\fk:= {\rm span}_\R \{iH_j, \, i(X_\al+X_{-\al}), \,  (X_\al-X_{-\al}), \,
\al \, \hbox{even}, \, (X_\be-iX_{-\be}), \, \be \, \hbox{odd} \}
\, \subset \, \fg
$$
Then $\fk$ is a compact real form of $\fg$, moreover 
$\fk$ consists of the fixed points of the involutory semiautomorphism:
$$
\begin{array}{cccc}
s:& \fg & \lra & \fg \\
& X_\al & \mapsto & -(i)^{p(\al)}X_{-\al} \\
& H_j & \mapsto & -H_j
\end{array}
$$
\end{theorem}

\begin{proof}
$\fk$ is a real subalgebra of $\fg$, it is fixed by $s$ and
it is a compact real form of $\fg$, by the ordinary theory.
\end{proof}

\begin{observation} \label{osp-counterex}
If $[X_\al,X_\al] \neq 0$, one can check right away that
$\fk$ written above is not closed for the bracket, so
the semiautomorphism $s$ makes no sense in this context.
Root vectors $X_\al$ with such property exist for example for $\rosp(V)$.
In this case there is no compact real form of $\fg$. In fact
let us look at $\rosp(1|2)$. The even part  
$\rosp(1|2)_0=\rsl_2(\C)$ acts on the odd part:
$\rosp(1|2)_1=\C^2 \oplus (\C^2)^*$, via the standard and contragredient
representation of  $\rsl_2(\C)$.
The compact real form of $\rsl_2(\C)$ is $\rsu(2)$, which
does not admit even dimensional real representations that complexified
give the $\rosp(1|2)_0$-module $\rosp(1|2)_1$.
This example shows that our hypothesis $[X_\al,X_\al] \neq 0$ leaves out
some superalgebras that do not admit compact real forms, among which 
the families of Lie superalgebras of classical type: $B$, $C$, $D$.
A similar argument excludes also the strange superalgebra $Q(2)$ and
type $P$.
\end{observation}

The existence of a compact real form $\fk$ of $\fg$ under
our hypotheses allows us to
proceed into the construction of the corresponding supergroup $K$,
which will be compact, since its underlying topological space is compact
and maximal, in the sense that the complexification of its
Lie superalgebra coincides with $\fg$. 
We shall do it in two different ways: via an algebraic method 
using the Chevalley supergroups theory 
and via the Super Harish Chandra Pairs theory.   
In the end, we will give a comparison between
these two different, but equivalent approaches.

\medskip

\section{Compact subgroups} \label{compact}

Let $G$ be a complex analytic Lie supergroup, with Lie superalgebra $\fg$
of classical type, subject to hypothesis (\ref{assumption}).
We want to introduce  $G^r$ the real supergroup
underlying $G$. The question of real
forms is treated in \cite{dm} with different language, but we shall neverthless
be inspired by the same philosophy.

\medskip

As before, let $G$ denote both the complex analytic supergroup and the unique
algebraic supergroup associated with it (see sec. \ref{analytic-algebraic}).

\medskip

The functor of points of the complex affine algebraic supergroup $G$
associates a group to each object in $\salg_\C$, the category of complex
commutative superalgebras: 
$$
G: \salg_\C \lra \grps, \qquad A \mapsto G(A)=\Hom(\cO(G), A)
$$
where $\cO(G)$ is the superalgebra of global sections of 
the structural sheaf of $G$.
Since all the supergroups we are interested in are subgroups of $\rGL(V)$
for a suitable superspace $V$, the functor $G$ is automatically defined
on the morphisms.
Similarly the functor of points of a real affine algebraic supergroup  
associates to any {\sl real} commutative superalgebra a group.
Now we want to understand what is the real algebraic supergroup $G^r$
{\sl underlying} $G$. 

\begin{definition} 
Define the following functor:
$$
G^r: \salg_\R \lra \grps, \qquad A \mapsto 
G^r(A)=\Hom_{\salg_\C}(\cO(G), A \otimes_\R \C )
$$
where $\salg_\R$ is the category of commutative $\R$-superalgebras.
We say that $G^r$ is the 
\textit{real algebraic supergroup underlying the complex
supergroup $G$}. 
\end{definition}


\medskip

In \cite{fg} and \cite{ccf} ch. 7,
we showed that the homogeneous one-parameter subgroups of $G$ 
come in three different
fashions.
Let $X_\al$, $X_\be$, $X_\ga$ be root vectors in the Chevalley basis,
$\al \in \Delta_0$, $\be$, $\ga$ $\in \Delta_1$, with
$[X_\be,X_\be]=0$, $[X_\ga,X_\ga] \neq 0$.
We define 
the following supergroup functors from the categories
of complex commutative superalgebras to the category of sets:
$$  
\begin{array}{rl}
x_\alpha(A) \!  &  := \big\{ x_\alpha(t)
:=\exp\!\big( t \, X_\alpha \big) \;\big|
\; t \in A_0 \,\big\} \; = \\ \\
& =\big\{ \big( 1 + t \, X_\alpha + 
t^2 \, {X_\alpha^2 \over 2}+ \cdots \big) \;\big| \; t \in A_0 \,\big\} 
\subset \rGL(V)(A), 
\\ \\ 
 x_\beta(A) \!  &  :=   
\big\{ x_\be(\theta):=\exp\!\big( \vartheta \, X_\beta \big)
\;\big|\; \vartheta \in A_1 \,\big\}  \; \\ \\
& =\big\{ \big( 1 +  \vartheta \, X_\beta \big) 
\;\big|\; \vartheta \in A_1 \,\big\} \subset  \rGL(V)(A), 
\\  \\
x_\gamma(A) \!  &  := \,  
\big\{x_\gamma(t,\theta):= \exp\!\big( \vartheta \, X_\gamma + 
t \, X_\gamma^{\,2} \big) \;\big|\; \vartheta \in A_1 \, , 
\, t \in A_0 \,\big\}  \, =  \\ \\
\phantom{\bigg|}  &  \phantom{:}= 
\,  \big\{ \big( 1+\vartheta \, X_\gamma\big) 
\exp\!\big( t \, X_\gamma^{\,2} \big) \;\big|\; 
\vartheta \in A_1 \, , \, t \in A_0 \,\big\}\subset  \rGL(V)(A). 
\end{array}  
$$
These functors are all representable and their representing Hopf
superalgebras are respectively: $\C[T]$, $\C[\Theta]$, $\C[T,\Theta]$.
The Hopf structure of the first two is trivial, while for the
last one see \cite{fg}. Notice that $x_\gamma$ will not appear in our $G$,
since our given hypothesis (\ref{assumption}).

\medskip

We want to understand what are the real
supergroups underlying $x_\al$,  $x_\beta$. 
Let us consider
the case $\be$ odd, $[X_\be,X_\be] = 0$ (the case $x_\al$, $\al$ even
is the same).
$$
x_\be^r: \salg_\R \lra \sets, \qquad
x_\be^r(A)=
\Hom(\C[\Theta], A \otimes_\R \C)
$$
Hence an element in $x_\be^r(A)$ 
is a morphism:
$$
\C[\Theta] \lra A \otimes_\R \C, \qquad 
\Theta \mapsto \theta:=\theta_0 \otimes 1 +  \theta_1 \otimes i
$$
We denote such morphism $x_\be^r(\theta)$ to stress the fact
that it is determined once we choose $\theta$ a
pair of elements $\theta_0, \theta_1 \in A_1$ the
odd part of the real commutative superalgebra $A$.
It makes then perfect sense to consider the element
$x_\be^r(\btheta)$:
$$
\C[\Theta] \lra A \otimes_\R \C, \qquad \Theta \mapsto 
\btheta:=\theta_0 \otimes 1 -  \theta_1 \otimes i
$$
where $\btheta$ is now associated to the pair of odd elements 
$(\theta_0,-\theta_1)$.

\medskip
We are ready for one of our main definitions.

\begin{definition} \label{sigmadef}
Let the notation be as above.
Define the following natural involution:
$$
\begin{array}{ccccc}
\sigma_A:& G^r(A)& \lra &G^r(A) & \\ \\
& x_\be(\theta)& \mapsto & x_{-\be}(-i\btheta) & \\ \\
& g & \mapsto & \sigma^0_A(g), & g \in G_0^r(A)
\end{array}
$$
where $\sigma^0:G_0^r \lra G_0^r$ is the involution of the
reduced group $G_0^r$ detailed in \cite{St}. 
\end{definition}


\begin{remark}
The ordinary involution $\sigma^0_A$ of $G_0^r(A)$ is such that
$\sigma^0_A( x_\al(t))=x_{-\al}(-\bt)$, for $\al$ an even root and
$\sigma^0_A( h(t) )=h(\bt^{-1})$ for $h$ a toral element (see \cite{St}). 
We cannot however define $\sigma^0_A$ by just specifying the images of
$x_\al(t)$,  $h(t)$ only, since these elements generate $G_0^r(A)$
only for $A$ {\sl local} (see \cite{sga3}). 
The existence of such $\sigma^0$ is however
granted by the ordinary theory. We do not incur into the same problem
in the super setting, since $G^r(A)$ is indeed generated by $G_0^r(A)$ and
$x_\be(A)$ for all superalgebras $A$ (see \cite{fg}).   
\end{remark}

We now wish to relate the two involutions $\sigma_A$ 
on $G^r(A)$ (see Def. \ref{sigmadef})
and $s$ on $\fg$ (see Theorem \ref{compactlie-thm}) and show they correspond
the same (semi) automorphism of $\cO(G)$, the Hopf superalgebra representing 
the supergroup $G$ with Lie superalgebra $\fg$.
We are deeply indebted to our referee for his suggestions regarding the next
observation.

\begin{observation}  \label{inv-obs}
Let $V$ be a complex super vector space (resp. a Lie superalgebra or
an Hopf superalgebra). We define ${\overline{V}}$ as the complex super vector
space where $c, v \mapsto cv$ is replaced by $c,v \mapsto {\bar c} v$
for $c \in \C$ (and similarly in the case of $V$ a Lie or Hopf superalgebra). 
Giving a $\C$-semilinear involution on $V$ is the same as
giving a $\C$-linear isomorphism $f:V \lra \overline{V}$,
such that $\overline{f} \circ f = {\mathrm {id}}_V$.

\smallskip

Consider now the $\C$-semilinear involution $s: \fg \lra \fg$ as in
\ref{compactlie-thm}. Such an involution corresponds to a $\C$-linear
isomorphism ${\hat s}:{\fg} \lra \overline{\fg}$, which extends
uniquely to give an isomorphism $\cU(\hat s):\cU(\fg) \lra  
\cU(\overline{\fg})$.
We wish to show that
such ${\hat s}$ corresponds to a unique Hopf superalgebra isomorphism
$\sigma':\overline{\cO(G)} \lra \cO(G)$ (equivalently to the
$\C$-semilinear involution $\sigma':\cO(G) \lra \cO(G)$)
inducing such ${\hat s}$ on the 
Lie superalgebras.

\smallskip

We need to go back briefly to the construction of Chevalley supergroups
(Ref. \cite{fg}). By our hypothesis $G$ is constructed via the Chevalley
supergroups recipe, in other words $G=G_V$, where $\rho:\fg \lra \rEnd_\C(V)$
is a faithful representation of $\fg$ into some complex super vector space
$V$. In the above notation, we can define immediately 
$\overline{\rho}: \overline{\fg} \lra
\overline{\rEnd_\C(V)}=\rEnd_\C(\overline{V})$, 
hence we obtain another algebraic supergroup
$G_{\overline{V}}$, through the Chevalley supergroup construction.
Since all of our objects are defined over the reals, we have that
$G_{\overline{V}}$ is represented by the Hopf superalgebra $\overline{\cO(G)}$.
So, in order to obtain the desired 
Hopf superalgebra isomorphism
$\sigma':\overline{\cO(G)} \lra \cO(G)$, it is enough that we give
an isomorphism $G_V \cong G_{\overline{V}}$. For this we
consider the faithful representation 
$\rho \cdot \hat s:\fg \lra \rEnd_\C({\overline{V}})$. Both $\rho$ and
$\rho \cdot \hat s$ have the same weights (up to sign), hence
according to the recipe, they will give an isomorphism 
$\psi:G_V\stackrel{\sim}\lra  G_{\overline{V}}$, ($\psi=\uspec(\sigma')$).

\smallskip

This implies the isomorphism 
$$
\psi_A:G_V(A \otimes_\R \C) \stackrel{\sim}\lra  
G_{\overline{V}}(A \otimes_\R \C), \qquad \forall A \in \salg_\R
$$ 
Consider now the isomorphism $c:G_{\overline{V}}(R) \cong G_{{V}}(\overline{R})$ 
obtained via the complex conjugation. 
Then we have:
$$
\sigma_A:G_V(A \otimes_\R \C) \, \stackrel{\uspec(\sigma')} 
\lra G_{\overline{V}}(A \otimes_\R \C) \,
\stackrel{c}\lra \, G_{{V}}(A \otimes_\R \C), \qquad A \in \salg_\R
$$
(since $A \in \salg_\R$ we have $\overline{A \otimes_\R \C}=A \otimes_\R \C$).
Notice that in our context $G^r(A)=G_V(A \otimes_\R \C)=\Hom_{\salg_\C}(\cO(G), 
A \otimes \C)$.
 
\smallskip

We now come to the uniqueness of $\sigma'$. Since $G$ is connected, we have
that $\cO(G)$ is naturally embedded into the Hopf dual $\cU(\fg)^o$
and similarly $\overline{\cO(G)} \subset \cU(\overline{\fg})^o$. In our
previous discussion we proved the existence of 
$\sigma':\overline{\cO(G)} \lra \cO(G)$ inducing $\hat s$, consequently
such $\sigma'$ must extend to give the morphism 
$\cU(\hat{s})^o :\cU(\fg)^o \lra  \cU(\overline{\fg})^o$ and this proves
the uniqueness.
 
\end{observation}

We are ready to define the compact form of our complex supergroup $G$.

\begin{definition} \label{Kdef}
We define the $K$, \textit{compact form} of $G$, as the algebraic supergroup 
represented by the Hopf superalgebra $\cO(G)^{\sigma'}$ (where
$\sigma'$ is as in
Observation \ref{inv-obs}):
$$
K: \salg_\R \lra \sets, \qquad K(A)=\Hom_{\salg_\R}(\cO(G)^{\sigma'},A)
$$
Notice that $\cO(G)^{\sigma'})$ is a real form of $\cO(G)$.
\end{definition}

The next proposition allows us to calculate the functor of
points of the supergroup $K$.

\begin{proposition}
Let the notation be as above. Then $K(A)$ coincides with the
$\sigma_A$ invariants in $G^r(A)$:
$$
K(A)\, = \, G^r(A)^{\sigma_A} \subset G^r(A)
$$
Furthermore $\Lie(K)=\fk$.
\end{proposition}

\begin{proof}
The first statement is a consequence of our definitions. In fact
since $\cO(G)=\cO(G)^\sigma \otimes_\R \C$, 
$$
\Hom_{\salg_\R}(\cO(G)^{\sigma'},A \otimes_\R \C)=G_V(A \otimes_\R \C)=G^r(A)
$$
In Obs. \ref{inv-obs} we established that $\sigma_A=c\circ \uspec(\sigma')$. 
Notice that $\sigma'$ is the identity on $\cO(G)^{\sigma'}$  
and $c$ is the automorphism induced by complex conjugation.
Hence $K(A)\, = \, G^r(A)^{\sigma_A}$.
As for the Lie superalgebras statement, it is enough to notice that
$\cU(s): \cU(\fg) \lra  \cU(\fg)$ corresponds to $\cU(\hat{s})$ as
in Obs. \ref{inv-obs}, the unique extension of $\sigma':\cO(G) \lra \cO(G)$
to $ \cU(\fg)^o$.
Hence the fixed points of $\sigma'$ correspond to  $\cU(\fg)^{\cU(s)}=\cU(\fk)$.
\end{proof}

We now consider few examples.

\begin{example}
1. $\rSL_2(\C)^r$. Let us write the involution $\sigma$ on the
big cell of $\rSL_2(\C)^r$, that is on $U^-TU^+$ where $U^\pm$ denote
the lower and upper unipotent subgroups, and $T$ the diagonal torus.
Such $\sigma$ will extend uniquely to the whole $\rSL_2(\C)^r$.
The big cell inside $\rSL_2(A)^r:=\rSL_2(\C)^r(A)$ is
$$
\begin{pmatrix} 1 & 0 \\ v & 1 \end{pmatrix}
\begin{pmatrix} t & 0 \\ 0 & t^{-1} \end{pmatrix}
\begin{pmatrix} 1 & u \\ 0 & 1 \end{pmatrix}=
\begin{pmatrix} t & tu \\ vt & utv+t^{-1} \end{pmatrix}
$$
Let's apply $\sigma$:
$$
\begin{array}{c}
\sigma \left( \begin{pmatrix} 1 & 0 \\ v & 1 \end{pmatrix}
\begin{pmatrix} t & 0 \\ 0 & t^{-1} \end{pmatrix}
\begin{pmatrix} 1 & u \\ 0 & 1 \end{pmatrix} \right)
=
\begin{pmatrix} 1 & -\bv \\ 0 & 1 \end{pmatrix} 
\begin{pmatrix} \bt^{-1} & 0 \\ 0 & \bt \end{pmatrix}
\begin{pmatrix} 1 & 0 \\ -\bu & 1 \end{pmatrix}= \\ \\
\begin{pmatrix} \bt^{-1}+\bu\bt\bv & -\bv\bt \\ 
-\bt \bu & \bt \end{pmatrix}
\end{array}
$$
Hence the condition to impose is:
$$
\begin{pmatrix} t & tv \\ tu & utv+t^{-1} \end{pmatrix}=
\begin{pmatrix} \bt^{-1}+\bu\bt\bv & -\bu\bt \\ 
-\bt \bu & \bt \end{pmatrix}
$$
Since any condition on the big cell and local real (super)algebras
will extend uniquely to $\rSL_2(A)^r$ we have:
$$
K(A):=(\rSL_2(A)^r)^\sigma=
\left\{ 
\begin{pmatrix} a & b \\ -\bar b & \bar a 
\end{pmatrix} \, \big| \, |a|^2+|b|^2=1 \right\}
$$
which is $SU(2)$ as expected.

\medskip

2. $\rGL(1|1)$. For $\rGL(1|1)$ we have global
coordinates and we can write the generic element $g$ of $\rGL(1|1)(A)$ as:
$$
g=\begin{pmatrix} 1 & 0 \\ \theta & 1 \end{pmatrix}
\begin{pmatrix} t & 0 \\ 0 & s \end{pmatrix}
\begin{pmatrix} 1 & \eta \\ 0 & 1 \end{pmatrix}=
\begin{pmatrix} t & t\eta \\ t\theta & \theta t \eta + s \end{pmatrix}=
\begin{pmatrix} a & \beta \\ \gamma & d \end{pmatrix}
$$
Applying $\sigma$ we get:
$$
\begin{array}{rl}
\sigma(g) 
& =
\begin{pmatrix} 1 & -i\btheta \\ 0 & 1 \end{pmatrix}
\begin{pmatrix} \bt^{-1} & 0 \\ 0 & \bs^{-1} \end{pmatrix}
\begin{pmatrix} 1 & 0 \\ -i\etab  & 1 \end{pmatrix} = 
\begin{pmatrix} \bt^{-1}-\btheta \bs^{-1} \etab & 
-i\bs^{-1}\btheta \\ -i\bs^{-1} \etab &  \bs^{-1} \end{pmatrix}
\end{array}
$$
Hence the condition to impose is:
$$
\begin{pmatrix} a & \beta \\ \gamma & d \end{pmatrix}:=
\begin{pmatrix} t & t\eta \\ t\theta & \theta t \eta + t \end{pmatrix}=
\begin{pmatrix} \bt^{-1}-\btheta \bs^{-1} \etab & 
-i\bs^{-1}\btheta \\ -i\bs^{-1} \etab &  \bs^{-1} \end{pmatrix}
$$
which gives after some straightforward calculations:
$$
\begin{array}{rl}
K(A)  = (\rGL(1|1)(A)^r)^\sigma & =
\left\{
\begin{pmatrix} a & \beta \\ \gamma & d \end{pmatrix} 
\, | \, a=\overline{a}^{-1}+\overline{\beta} d^{-1} 
\overline{\gamma} \overline{a}^{-2},  \,
d^{-1} = \overline{d} - \overline{\beta} 
\overline{a}^{-1} \overline{\gamma}, \right.
\\ \\ 
& \left. \beta= - i \overline{a}^{-1} \overline{\gamma} d, \,
\gamma= -id \overline{\beta}  \overline{a}^{-1}  
\right\}
\end{array}
$$

\medskip

3. $\rSL(1|1)$. For $\rSL(1|1)$ we have global
coordinates and we can write the generic element $g$ of $\rSL(1|1)(A)$ as:
$$
g=\begin{pmatrix} 1 & 0 \\ \theta & 1 \end{pmatrix}
\begin{pmatrix} t & 0 \\ 0 & t \end{pmatrix}
\begin{pmatrix} 1 & \eta \\ 0 & 1 \end{pmatrix}=
\begin{pmatrix} t & t\eta \\ t\theta & \theta t \eta + t \end{pmatrix}
$$
Applying $\sigma$ we get:
$$
\begin{array}{rl}
\sigma(g)
=\begin{pmatrix} 1 & -i\btheta \\ 0 & 1 \end{pmatrix}
\begin{pmatrix} \bt^{-1} & 0 \\ 0 & \bt^{-1} \end{pmatrix}
\begin{pmatrix} 1 & 0 \\ -i\etab  & 1 \end{pmatrix} 
=\begin{pmatrix} \bt^{-1}-\btheta \bt^{-1} \etab & 
-i\bt^{-1}\btheta \\ -i\bt^{-1} \etab &  \bt^{-1} \end{pmatrix}
\end{array}
$$
Hence the condition to impose is:
$$
\begin{pmatrix} a & \beta \\ \gamma & d \end{pmatrix}:=
\begin{pmatrix} t & t\eta \\ t\theta & \theta t \eta + t \end{pmatrix}=
\begin{pmatrix} \bt^{-1}-\btheta \bt \etab & 
-i\bt^{-1}\btheta \\ -i\bt^{-1} \etab &  \bt^{-1} \end{pmatrix}
$$
which gives in terms of the entries $a$, $d$, $\beta$, $\gamma$:
$$
K(A) = (\rSL(1|1)(A)^r)^\sigma=:SU(1|1)=
\left\{
\begin{pmatrix} a & \beta \\ -{\bar \beta} & {\bar a}^{-1} \end{pmatrix} 
\, | \, {\bar a}a(1+i\beta{\bar \beta})=1 \, \right\}
$$
Notice that the reduced part is a torus as one expects.

\medskip

4. $\rSL(m|n)$. The calculation resembles very much (2), this
time the letters are matrices.
$$
\begin{array}{rl}
\sigma \left(
\begin{pmatrix} 1 & 0 \\ \theta & 1 \end{pmatrix} 
\begin{pmatrix} t & 0 \\ 0 & s \end{pmatrix}
\begin{pmatrix} 1 & \eta \\ 0 & 1 \end{pmatrix}\right) & =
\begin{pmatrix} 1 & -i\btheta^t \\ 0 & 1 \end{pmatrix}
\begin{pmatrix} (\bt^{-1})^t & 0 \\ 0 & (\bs^{-1})^t \end{pmatrix}
\begin{pmatrix} 1 & 0 \\ -i\etab^t  & 1 \end{pmatrix} = \\ \\
&=\begin{pmatrix} (\bt^{-1})^t-\btheta^t (\bs^{-1})^t \etab^t & 
-i\btheta^t (\bs^{-1})^t \\ -i(\bs^{-1})^t \etab^t &  (\bs^{-1})^t \end{pmatrix}
\end{array}
$$
The condition to impose is: 
$$
\begin{pmatrix} a & \beta \\ \gamma & d \end{pmatrix}:=
\begin{pmatrix} t & t\eta \\ \theta t & \theta t \eta + s \end{pmatrix}=
\begin{pmatrix} (\bt^{-1})^t-\btheta^t (\bs^{-1})^t \etab^t & 
-i\btheta^t (\bs^{-1})^t \\ -i(\bs^{-1})^t \etab^t &  (\bs^{-1})^t \end{pmatrix}
$$
which gives:
$$
K(A) = (\rSL(m|n)(A)^r)^\sigma=:SU(m|n)=
\left\{ \begin{pmatrix} a & \beta \\ \gamma & d \end{pmatrix} \right\}
$$
with
$$
\begin{array}{cc}
a=({\bar a}^{-1})^t- (\overline{a}^{-1})^t\overline{\gamma}^t
d\overline{\beta}^t (\overline{a}^{-1})^t, & 
\beta= - i({\bar a}^{-1})^t {\bar \gamma}^t d \\ \\
\gamma= - id {\bar \beta}^t  ({\bar a}^{-1})^t &
d=({\bar d}^{-1})^t-\gamma a^{-1} \beta
\end{array}
$$
\end{example}

\section{The SHCP's approach to compact forms}

We recall that we have a SHCP when we are given a pair $(G_0, \fg)$
consisting of an ordinary Lie group and a Lie superalgebra, such
that $\Lie(G_0)=\fg_0$ and $G_0$ acts on $\fg$ in such a way 
that the differential of the action is the Lie superbracket in $\fg$
(Ref. \cite{kos}).
The category of SHCP's is equivalent to the category of super Lie groups.
In \cite{cf} such equivalence is extended to the categories of
analytic and algebraic supergroups, with the obvious changes in the
definition of SHCP (see also \cite{mas} for a more comprehensive
treatment of this equivalence).

\medskip

If $G$ is a complex (analytic or algebraic) supergroup with Lie
superalgebra of classical type and
with $G_0$ the reduced underlying ordinary
group, $(G_0, \fg)$ is a SHCP, where $\fg=\Lie(G)$.

\medskip

We now turn to our setting. 
Let $G$ be a complex analytic Lie supergroup, with Lie superalgebra $\fg$
of classical type, subject to hypothesis (\ref{assumption}).
Let $K_0 \subset G_0$ be the maximal compact Lie subgroup of $G_0$
and $\fk$ as in \ref{defcompactliealg}.
Then $K_{shcp}=(K_0, \fk)$ is a SHCP. In fact by the ordinary
theory we have that $\fk_0=\Lie(K_0)$, moreover the natural
adjoint action of $K_0$ on $\fk$ obtained by restricting the
adjoint action of $G_0$ on $\fg$ induces the bracket. 

\begin{definition}
Let the notation and assumptions be as above. 
We define $K_{shcp}=(K_0, \fk)$ the \textit{shcp maximal
compact} in $G$.
\end{definition}

The supergroup $K_{shcp}$ corresponds, via the above
mentioned equivalence of categories, to the maximal compact 
subgroup $K$ as defined in the previous sections. In fact they both
have the same underlying Lie group $K_0$, the same Lie superalgebra
$\fk$ and $K_0$ acts on $\fk$ in the same way in both cases. 
We now want to give a more stringent characterization of the
superalgebra of the global sections on $K$ so to establish 
a link between the SHCP's
and the functor of points approach detailed in the previous section.

\medskip

We recall that if $(G_0, \fg)$ is a SHCP, we can obtain the
global sections of the corresponding algebraic supergroup $G$ as follows:
$$
\cO(G_0, \fg)=
\Hom_{\cU(\fg_0)}(\cU(\fg), \cO(G_0))
$$
(for more details on these definitions see \cite{ccf} ch. 7, sec. 7.4).
The equivalence of the category of SHCP and algebraic supergroups
prescribes that $\cO(G_0, \fg) \cong \cO(G)$ through the isomorphism:
$$ 
\begin{array}{ccc}
\cO(G) & \lra &  \cO(G_0, \fg) \\ \\
s & \mapsto & X \mapsto (-1)^{|X|}|D_Xs| 
\end{array}
$$
where $D_X$ is the left invariant differential operator associated with
$X \in \cU(\fg)$.

\medskip

Let $s$ be the involution of the complex simple Lie
superalgebra $\fg$ (as in \ref{compactlie-thm})
such that $\fg^s=\fk$. 
The involution $s$ induces an involution on the universal
enveloping algebra (still denoted with $s$)
and we have $\cU(\fg)^{s} = \cU(\fk)$. 

\medskip

We define the following subalgebra
of $\cO(G_0, \fg)$:
$$
\begin{array}{rl}
\cO(G_0, \fg)^{s,\sigma_0} =\{f : \cU(\fg) \lra \cO(G_0)  |  
f \cdot s=f, \, \sigma_0 \cdot f(u)=f(u) \} 
\subset \cO(G_0, \fg)
\end{array}
$$

\begin{proposition}
Let the notation be as above. Then:
$$
\cO(K_0,\fk)=\cO(G_0, \fg)^{s,\sigma_0}.
$$
\end{proposition}

\begin{proof}
Define
$$
\begin{array}{cccc}
\psi: & \cO(G_0, \fg)^{s,\sigma_0} & \lra & \cO(K_0,\fk) \\ \\
& f & \mapsto & f|_{\cU(\fk)}
\end{array}
$$
Since $f(u) = \sigma_0 \cdot f(u)$ we have that $\psi$ is well defined, i.e.
$f(u) \in \cO(K_0)= \cO(G_0)^{\sigma_0}$. Furthermore $\psi$ is a superalgebra
morphism. Now the fact $\psi$ is injective. Assume $f|_{\cU(\fk)}=0$. Since
$f \circ s=f$ we have $f(s(u))=f(u)$. $u+s(u)$ is invariant, hence
$f(u+s(u))=2f(u)=0$ for all $u \in \cU(\fg)$. Now the 
surjectivity. Let $g \in \cO(K_0,\fk)$, we want to determine $f \in 
\cO(G_0, \fg)^{s,\sigma_0}$ so that $f|_{\cU(\fk)}=g$. If such $f$ exists, it
must coincide with $g$ on the $s$ invariant elements of $\cU(\fg)$. Hence:
$g(s(u)+u)=f(s(u)+u)=2f(u)$. Hence we define: $f=(1/2)g(s(u)+u)$.
\end{proof}

The next corollary establishes a direct link between the two
approaches to the definition of $K$ maximal compact in $G$.

\begin{corollary}
Let the notation be as above. Then
$$
\cO(K)=\cO(G)^{\sigma'} \cong  \cO(G_0, \fg)^{s,\sigma_0} = \cO(K_0,\fk)=\cO(K_{shcp})
$$
in other words we have an explicit isomorphism between the superalgebras
of global sections of  the maximal compact subgroup $K$ of $G$ defined as in
\ref{Kdef} and the maximal compact sugroup $K_{shcp}$ corresponding
to the SHCP $(K_0,\fk)$. 
\end{corollary}

\end{document}